\documentclass{amsart}
\usepackage{amssymb}
\usepackage{amsmath}
\usepackage{amsfonts}

\newtheorem{theorem}{Theorem}
\theoremstyle{plain}

\newtheorem{corollary}{Corollary}

\newtheorem{definition}{Definition}

\newtheorem{proposition}{Proposition}
\newtheorem{remark}{Remark}

\numberwithin{equation}{section}
\input{tcilatex}

\begin{document}
\title[Existence of Weak Solutions for $p\left( .\right) $-Laplacian Equation%
]{Existence of Weak Solutions for $p\left( .\right) $-Laplacian Equation via
Compact Embeddings of the Double Weighted Variable Exponent Sobolev Spaces}
\author{Cihan UNAL}
\address{Sinop University\\
Faculty of Arts and Sciences\\
Department of Mathematics}
\email{cihanunal88@gmail.com}
\urladdr{}
\thanks{}
\author{Ismail AYDIN}
\address{Sinop University\\
Faculty of Arts and Sciences\\
Department of Mathematics}
\email{iaydin@sinop.edu.tr}
\urladdr{}
\thanks{}
\dedicatory{}
\thanks{}

\begin{abstract}
In this study, we define double weighted variable exponent Sobolev spaces $%
W^{1,q(.),p(.)}\left( \Omega ,\vartheta _{0},\vartheta \right) $ with
respect to two different weight functions. Also, we investigate the basic
properties of this spaces. Moreover, we discuss the existence of weak
solutions for weighted Dirichlet problem of $p(.)$-Laplacian equation%
\begin{equation*}
\left\{ 
\begin{array}{cc}
-\func{div}\left( \vartheta (x)\left\vert \nabla f\right\vert
^{p(x)-2}\nabla f\right) =\vartheta _{0}(x)\left\vert f\right\vert ^{q(x)-2}f
& x\in \Omega \\ 
f=0 & x\in \partial \Omega%
\end{array}%
\right.
\end{equation*}%
under some conditions of compact embedding involving the double weighted
variable exponent Sobolev spaces.
\end{abstract}

\maketitle

\section{Introduction}

The history of potential theory begins in 17th century. Its development can
be traced to such greats as Newton, Euler, Laplace, Lagrange, Fourier,
Green, Gauss, Poisson, Dirichlet, Riemann, Weierstrass, Poincar\'{e}. We
refer to the book by Kellogg \cite{Kel} for references to some of the old
works.

Kov\'{a}\v{c}ik and R\'{a}kosn\'{\i}k \cite{K} introduced the variable
exponent Lebesgue space $L^{p\left( .\right) }(%
\mathbb{R}
^{d})$ and the Sobolev space $W^{k,p(.)}\left( 
\mathbb{R}
^{d}\right) $. They present some basic properties of the variable exponent
Lebesgue space $L^{p\left( .\right) }(%
\mathbb{R}
^{d})$ and the Sobolev space $W^{k,p(.)}\left( 
\mathbb{R}
^{d}\right) $ such as reflexivity and H\"{o}lder inequalities were obtained.
Also, Fan and Zhao \cite{Fan} present important results for the variable
exponent Lebesgue and Sobolev spaces. The study of electrorheological fluids
is one of the important areas where these spaces have found applications,
see \cite{Ruz}. As an another area, we can say the study of variational
integrals with non-standard growth, see \cite{Ac}, \cite{Zhi}. The
boundedness of the maximal operator was an open problem in $L^{p\left(
.\right) }(%
\mathbb{R}
^{d})$ for a long time. Diening \cite{Dien} proved the first time this state
over bounded domains if $p\left( .\right) $ satisfies locally log-H\"{o}lder
continuous condition, that is,%
\begin{equation*}
\left\vert p\left( x\right) -p\left( y\right) \right\vert \leq \frac{C}{-\ln
\left\vert x-y\right\vert }\text{, }x,y\in \Omega ,\text{ }\left\vert
x-y\right\vert \leq \frac{1}{2}
\end{equation*}%
where $\Omega $ is a bounded domain. We denote by $P^{\log }\left( 
\mathbb{R}
^{d}\right) $ the class of variable exponents which satisfy the log-H\"{o}%
lder continuous condition. Diening later extended the result to unbounded
domains by supposing, in addition, that the exponent $p\left( .\right) =p$
is a constant function outside a large ball. After this study, many
absorbing and crucial papers appeared in non-weighted and weighted variable
exponent spaces. For a historical journey, we refer \cite{Cruz}, \cite{Dien4}
and references therein.

The operator $-\Delta _{p\left( .\right) }f=-\func{div}\left( \left\vert
\nabla f\right\vert ^{p\left( .\right) -2}\nabla f\right) $ is called $%
p\left( .\right) $-Laplacian. The study of differential equations and
variational problems with $p\left( .\right) $-growth conditions arouses much
interest with the development of elastic mechanics, electrorheological fluid
dynamics and image processing etc. We refer the readers \cite{Gold}, \cite%
{Kim}, \cite{MasO}, \cite{Ohno}, \cite{Sai}, \cite{UnalAy} and references
therein. In general, the methods used in these works are base on continuous
and compact embeddings between Lebesgue and Sobolev spaces.

In 2003, Fan and Zhang obtained a weak solution in $W_{0}^{1,p\left(
.\right) }\left( \Omega \right) $ to the Dirichlet problem of $p\left(
.\right) $-Laplacian%
\begin{equation*}
\left\{ 
\begin{array}{cc}
-\func{div}\left( \left\vert \nabla u\right\vert ^{p\left( x\right)
-2}\nabla u\right) =f\left( x,u\right) & x\in \Omega \\ 
u=0 & x\in \partial \Omega%
\end{array}%
\right.
\end{equation*}%
where $f:\Omega \times 
\mathbb{R}
\longrightarrow 
\mathbb{R}
$ is a Carath\'{e}odory function which satisfies the growth condition, see 
\cite{FanZ}. Moreover, in recent years, $p\left( .\right) $-Laplacian
equations and variational problems with $p\left( .\right) $-growth
conditions have been studied by several authors, see \cite{Bui}, \cite{FanZ}%
, \cite{Ilias}, \cite{Kim}, \cite{LiuLiu}, \cite{MasO}.

In \cite{Gold}, the authors deal with two Dirichlet boundary value problems
involving the weighted $p$-Laplacian. They give existence and multiplicity
results under suitable conditions\ in constant exponent case. One of the our
purpose is to extend to variable case of some of the results in \cite{Gold}.

In this study, we present and investigate double weighted variable exponent
Sobolev spaces $W^{1,q(.),p(.)}\left( \Omega ,\vartheta _{0},\vartheta
\right) $ with respect to two different weight functions. The main purpose
of this paper is to study the existence of weak solutions of $p(.)$%
-Laplacian problem%
\begin{equation}
\left\{ 
\begin{array}{cc}
-\func{div}\left( \vartheta (x)\left\vert \nabla f\right\vert
^{p(x)-2}\nabla f\right) =\vartheta _{0}(x)\left\vert f\right\vert ^{q(x)-2}f
& x\in \Omega \\ 
f=0 & x\in \partial \Omega%
\end{array}%
\right.  \label{P1}
\end{equation}%
for $f\in W_{0}^{1,q(.),p(.)}\left( \Omega ,\vartheta _{0},\vartheta \right) 
$ where $\Omega \subset 
\mathbb{R}
^{d}$ is a bounded domain. Moreover, we discuss the necessary conditions for
existence of weak solutions for (\ref{P1}) involving the Poincar\'{e}
inequality in $W^{1,q(.),p(.)}\left( \Omega ,\vartheta _{0},\vartheta
\right) $, several continuous and compact embeddings.

\section{Notation and Preliminaries}

In this paper, we will work on $\Omega $ with Lebesgue measure $dx$. Also,
the elements of the space $C_{0}^{\infty }\left( \Omega \right) $ are the
infinitely differentiable functions with compact support. A normed space $%
\left( X,\left\Vert .\right\Vert _{X}\right) $ is called a Banach function
space (shortly BF- space), if Banach space $\left( X,\left\Vert .\right\Vert
_{X}\right) $ is continuously embedded into $L_{loc}^{1}\left( \Omega
\right) ,$ briefly $X\hookrightarrow L_{loc}^{1}\left( \Omega \right) ,$
i.e. for any compact subset $K\subset \Omega $ there is some constant $%
c_{K}>0$ such that $\left\Vert f\chi _{K}\right\Vert _{L^{1}\left( \Omega
\right) }\leq c_{K}\left\Vert f\right\Vert _{X}$ for every $f\in X.$
Moreover, a normed space $X$ is compactly embedded in a normed space $Y,$
briefly $X\hookrightarrow \hookrightarrow Y,$ if $X\hookrightarrow Y$ and
the identity operator $I:X\longrightarrow Y$ is compact, equivalently, $I$
maps every bounded sequence $\left( x_{i}\right) _{i\in 
\mathbb{N}
}$ into a sequence $\left( I\left( x_{i}\right) \right) _{i\in 
\mathbb{N}
}$ that contains a subsequence converging in $Y.$ Suppose that $X$ and $Y$
are two Banach spaces and $X$ is reflexive. Then $I:X\longrightarrow Y$ is a
compact operator if and only if $I$ maps weakly convergent sequences in $X$
onto convergent sequences in $Y.$ More details can be found in \cite{Ada}.

Let $\Omega \subset 
\mathbb{R}
^{d}$ is bounded and $\vartheta $ is a weight function$.$ It is known that a
function $f\in C_{0}^{\infty }\left( \Omega \right) $ satisfies Poincar\'{e}
inequality in $L_{\vartheta }^{1}(\Omega )$ if and only if the inequality%
\begin{equation*}
\dint\limits_{\Omega }\left\vert f(x)\right\vert \vartheta \left( x\right)
dx\leq c\left( diam\left( \Omega \right) \right) \dint\limits_{\Omega
}\left\vert \nabla f(x)\right\vert \vartheta \left( x\right) dx
\end{equation*}%
holds, see \cite{Hei}.

\begin{definition}
Let $\Omega \subset 
\mathbb{R}
^{d},$ $d\geq 1,$ be a domain with non-empty boundary $\partial \Omega $,
denote%
\begin{equation*}
L_{+}^{\infty }\left( \Omega \right) =\left\{ p\left( .\right) \in L^{\infty
}\left( \Omega \right) :\underset{x\in \Omega }{\text{essinf}}p(x)>1\right\}
\end{equation*}%
For a measurable function $p\left( .\right) :\Omega \longrightarrow \lbrack
1,\infty )$ (called a variable exponent on $\Omega $) by the symbol $P\left(
\Omega \right) $. In this paper, the function $p(.)$ always denotes a
variable exponent. For $p\left( .\right) \in P\left( \Omega \right) ,$ we
put 
\begin{equation*}
p^{-}=\underset{x\in \Omega }{\text{essinf}}p(x)\text{, \ \ \ \ \ \ }p^{+}=%
\underset{x\in \Omega }{\text{esssup}}p(x)\text{.}
\end{equation*}%
The variable exponent Lebesgue spaces $L^{p(.)}(\Omega )$ consist of all
measurable functions $f$ such that $\varrho _{p(.)}(\lambda f)<\infty $ for
some $\lambda >0$, equipped with the Luxemburg norm%
\begin{equation*}
\left\Vert f\right\Vert _{L^{p(.)}(\Omega )}=\inf \left\{ \lambda >0:\varrho
_{p(.)}\left( \frac{f}{\lambda }\right) \leq 1\right\} \text{,}
\end{equation*}%
where 
\begin{equation*}
\varrho _{p(.)}(f)=\dint\limits_{\Omega }\left\vert f(x)\right\vert ^{p(x)}dx%
\text{.}
\end{equation*}%
Let $p^{+}<\infty $. Then $f\in L^{p(.)}(\Omega )$ if and only if $\varrho
_{p(.)}(f)<\infty $. The space $L^{p(.)}\left( \Omega \right) $ is a Banach
space with respect to $\left\Vert .\right\Vert _{L^{p(.)}(\Omega )}$. If $%
p\left( .\right) =p$ is a constant function, then the norm $\left\Vert
.\right\Vert _{L^{p(.)}(\Omega )}$ coincides with the usual Lebesgue norm $%
\left\Vert .\right\Vert _{p}$, see \cite{K}. In this paper, we assume that
all variable exponents are belong to $L_{+}^{\infty }\left( \Omega \right) $.
\end{definition}

\begin{definition}
A measurable and locally integrable function $\vartheta :\Omega
\longrightarrow \left( 0,\infty \right) $ is called a weight function. We
say that $\vartheta _{1}\prec \vartheta _{2}$ if only if there exists $c>0$
such that $\vartheta _{1}(x)\leq c\vartheta _{2}(x)$ for all $x\in \Omega $.
Now, we denote%
\begin{equation*}
W\left( \Omega \right) =\left\{ \vartheta \in L_{loc}^{1}\left( \Omega
\right) :\vartheta >0\text{ almost everywhere in }\Omega \right\} .
\end{equation*}%
For $p\left( .\right) \in L_{+}^{\infty }\left( \Omega \right) ,$ we define%
\begin{equation*}
W_{p(.)}\left( \Omega \right) =\left\{ \vartheta \in W\left( \Omega \right)
:\vartheta ^{-\frac{1}{p(.)-1}}\in L_{loc}^{1}\left( \Omega \right) \right\}
.
\end{equation*}%
Moreover, for $p\left( .\right) \in L_{+}^{\infty }\left( \Omega \right) $
and $\vartheta \in W\left( \Omega \right) $, we consider the weighted
variable exponent Lebesgue space%
\begin{equation*}
L^{p(.)}(\Omega ,\vartheta )=\left\{ f\left\vert f:\Omega \longrightarrow 
\mathbb{R}
\text{ measurable and }\dint\limits_{\Omega }\left\vert f(x)\right\vert
^{p(x)}\vartheta (x)dx<\infty \right. \right\}
\end{equation*}%
with the Luxemburg norm 
\begin{equation*}
\left\Vert f\right\Vert _{L^{p(.)}(\Omega ,\vartheta )}=\inf \left\{ \lambda
>0:\varrho _{p(.),\vartheta }\left( \frac{f}{\lambda }\right)
=\dint\limits_{\Omega }\left\vert \frac{f(x)}{\lambda }\right\vert
^{p(x)}\vartheta (x)dx\leq 1\right\} .
\end{equation*}%
The space $L^{p(.)}(\Omega ,\vartheta )$ is a Banach space with respect to $%
\left\Vert .\right\Vert _{L^{p(.)}(\Omega ,\vartheta )}.$ Moreover, $f\in
L^{p(.)}(\Omega ,\vartheta )$ if and only if $\left\Vert f\right\Vert
_{L^{p(.)}(\Omega ,\vartheta )}=\left\Vert f\vartheta ^{\frac{1}{p(.)}%
}\right\Vert _{L^{p(.)}(\Omega )}<\infty $. It is known that we have the
relationships between the modular $\varrho _{p(.),\vartheta }(.)$ and the
norm $\left\Vert .\right\Vert _{L^{p(.)}(\Omega ,\vartheta )}$ are as follows%
\begin{equation*}
\min \left\{ \varrho _{p(.),\vartheta }(f)^{\frac{1}{p^{-}}},\varrho
_{p(.),\vartheta }(f)^{\frac{1}{p^{+}}}\right\} \leq \left\Vert f\right\Vert
_{L^{p(.)}(\Omega ,\vartheta )}\leq \max \left\{ \varrho _{p(.),\vartheta
}(f)^{\frac{1}{p^{-}}},\varrho _{p(.),\vartheta }(f)^{\frac{1}{p^{+}}%
}\right\}
\end{equation*}%
and%
\begin{equation*}
\min \left\{ \left\Vert f\right\Vert _{L^{p(.)}(\Omega ,\vartheta
)}^{p^{-}},\left\Vert f\right\Vert _{L^{p(.)}(\Omega ,\vartheta
)}^{p^{+}}\right\} \leq \varrho _{p(.),\vartheta }(f)\leq \max \left\{
\left\Vert f\right\Vert _{L^{p(.)}(\Omega ,\vartheta )}^{p^{-}},\left\Vert
f\right\Vert _{L^{p(.)}(\Omega ,\vartheta )}^{p^{+}}\right\}
\end{equation*}%
Also, if $0<C\leq \vartheta $, then we have $L^{p(.)}(\Omega ,\vartheta
)\hookrightarrow L^{p(.)}(\Omega ),$ since one easily sees that 
\begin{equation*}
C\dint\limits_{\Omega }\left\vert f(x)\right\vert ^{p(x)}dx\leq
\dint\limits_{\Omega }\left\vert f(x)\right\vert ^{p(x)}\vartheta (x)dx
\end{equation*}%
and $C\left\Vert f\right\Vert _{L^{p(.)}(\Omega )}\leq \left\Vert
f\right\Vert _{L^{p(.)}(\Omega ,\vartheta )}$. Moreover, the dual space of $%
L^{p(.)}(\Omega ,\vartheta )$ is $L^{p^{\prime }(.)}(\Omega ,\vartheta
^{\ast })$, where $\frac{1}{p(.)}+\frac{1}{p^{\prime }(.)}=1$ and $\vartheta
^{\ast }=\vartheta ^{1-p^{\prime }\left( .\right) }=\vartheta ^{-\frac{1}{%
p(.)-1}}.$ For more details, we refer \cite{A1}, \cite{A2} and \cite{Kok}.
\end{definition}

\begin{theorem}
(see \cite{A2})If $\vartheta \in W_{p(.)}\left( \Omega \right) $, then $%
L^{p(.)}(\Omega ,\vartheta )\hookrightarrow L_{loc}^{1}\left( \Omega \right)
\hookrightarrow D^{\shortmid }(\Omega )$, that is, every function in $%
L^{p(.)}(\Omega ,\vartheta )$ has distributional (weak) derivative, where $%
D^{\shortmid }(\Omega )$ is distribution space.
\end{theorem}

\begin{remark}
\label{embed}(see \cite{Kuf})If $\vartheta \notin W_{p(.)}\left( \Omega
\right) $, then the embedding $L^{p(.)}(\Omega ,\vartheta )\hookrightarrow
L_{loc}^{1}\left( \Omega \right) $ need not hold.
\end{remark}

Remark \ref{embed} says that the assumption $\vartheta \in W_{p(.)}\left(
\Omega \right) $ is necessary for distributional (weak) derivative
techniques.

\begin{proposition}
Assume that $\vartheta \in W_{p(.)}\left( \Omega \right) $, $\phi \in
C_{0}^{\infty }\left( \Omega \right) $. Also, let a multi-index $\alpha \in 
\mathbb{N}
_{0}^{d}$ be fixed. Then, the formula%
\begin{equation*}
L_{\alpha }(f)=\dint\limits_{\Omega }fD^{\alpha }\phi dx,\text{ }f\in
L^{p(.)}(\Omega ,\vartheta )
\end{equation*}%
defines a continuous linear functional $L_{\alpha }$ on $L^{p(.)}(\Omega
,\vartheta )$ where $C_{0}^{\infty }\left( \Omega \right) $ is the space of $%
C^{\infty }\left( \Omega \right) $ functions with compact support in $\Omega 
$.
\end{proposition}

\begin{proof}
If we denote $Q=$supp$\phi $, then we have $Q=\overline{Q}$. By the H\"{o}%
lder inequality for $L^{p(.)}(\Omega ,\vartheta ),$ we get%
\begin{eqnarray*}
\left\vert L_{\alpha }(f)\right\vert &\leq &\dint\limits_{\Omega }\left\vert
f\right\vert \vartheta ^{\frac{1}{p(.)}}\vartheta ^{-\frac{1}{p(.)}%
}\left\vert D^{\alpha }\phi \right\vert dx \\
&\leq &c\left\Vert f\vartheta ^{\frac{1}{p(.)}}\right\Vert _{L^{p(.)}(\Omega
)}\left\Vert \vartheta ^{-\frac{1}{p(.)}}D^{\alpha }\phi \right\Vert
_{L^{q(.)}(\Omega )} \\
&\leq &C\left\Vert f\right\Vert _{L^{p(.)}(\Omega ,\vartheta )}<\infty .
\end{eqnarray*}%
where $\frac{1}{p(.)}+\frac{1}{q(.)}=1.$
\end{proof}

\begin{definition}
We set the weighted variable Sobolev spaces $W^{k,p(.)}\left( \Omega
,\vartheta \right) $ by%
\begin{equation*}
W^{k,p(.)}\left( \Omega ,\vartheta \right) =\left\{ f\in L^{p(.)}(\Omega
,\vartheta ):D^{\alpha }f\in L^{p(.)}(\Omega ,\vartheta ),0\leq \left\vert
\alpha \right\vert \leq k\right\}
\end{equation*}%
equipped with the norm 
\begin{equation*}
\left\Vert f\right\Vert _{W^{k,p(.)}\left( \Omega ,\vartheta \right)
}=\dsum\limits_{0\leq \left\vert \alpha \right\vert \leq k}\left\Vert
D^{\alpha }f\right\Vert _{L^{p(.)}(\Omega ,\vartheta )}
\end{equation*}%
where $\alpha \in 
\mathbb{N}
_{0}^{d}$ is a multi-index, $\left\vert \alpha \right\vert =\alpha
_{1}+\alpha _{2}+...+\alpha _{d}$ and $D^{\alpha }=\frac{\partial
^{\left\vert \alpha \right\vert }}{\partial _{x_{1}}^{\alpha
_{1}}...\partial _{x_{d}}^{\alpha _{d}}}$. It can be shown that $%
W^{k,p(.)}\left( \Omega ,\vartheta \right) $ is a reflexive Banach space. In
particular, the space $W^{1,p(.)}\left( \Omega ,\vartheta \right) $ is
defined by 
\begin{equation*}
W^{1,p(.)}\left( \Omega ,\vartheta \right) =\left\{ f\in L^{p(.)}(\Omega
,\vartheta ):\left\vert \nabla f\right\vert \in L^{p(.)}(\Omega ,\vartheta
)\right\} .
\end{equation*}%
The function $\varrho _{1,p(.),\vartheta }:W^{1,p(.)}\left( \Omega
,\vartheta \right) \longrightarrow \left[ 0,\infty \right) $ is shown as $%
\varrho _{1,p(.),\vartheta }(f)=\varrho _{p(.),\vartheta }(f)+\varrho
_{p(.),\vartheta }\left( \nabla f\right) $. Also, the norm $\left\Vert
f\right\Vert _{W^{1,p(.)}\left( \Omega ,\vartheta \right) }=\left\Vert
f\right\Vert _{L^{p(.)}(\Omega ,\vartheta )}+\left\Vert \nabla f\right\Vert
_{L^{p(.)}(\Omega ,\vartheta )}$ makes the space $W^{1,p(.)}\left( \Omega
,\vartheta \right) $ a Banach space. If the exponent $p\left( .\right) $
satisfies locally log-H\"{o}lder continuous condition, then a lot of
regularities for variable exponent spaces holds. Because, the space $%
C_{0}^{\infty }\left( \Omega \right) $ is dense in $W^{1,p(.)}\left( \Omega
,\vartheta \right) $ under the circumstances, see \cite{A2}. More
information on the classic theory of variable exponent spaces can be found
in \cite{Dien4} and \cite{K}.
\end{definition}

Throughout this paper, we assume that $1<q^{-}\leq q\left( .\right) \leq
q^{+}<p^{-}\leq p\left( .\right) \leq p^{+}<\lambda <\infty $, $\vartheta
\in W_{p(.)}\left( \Omega \right) $ and $\vartheta _{0}\in W_{q(.)}\left(
\Omega \right) $ where $\Omega \subset 
\mathbb{R}
^{d}$ is a bounded domain.

\section{Main Results}

Before we consider the existence of weak solutions of (\ref{P1}), we present
and investigate the double weighted variable exponent Sobolev spaces.

\begin{definition}
The double weighted variable exponent Sobolev spaces $W^{1,q(.),p(.)}\left(
\Omega ,\vartheta _{0},\vartheta \right) $ is defined by%
\begin{equation*}
W^{1,q(.),p(.)}\left( \Omega ,\vartheta _{0},\vartheta \right) =\left\{ f\in
L^{q(.)}(\Omega ,\vartheta _{0}):\left\vert \nabla f\right\vert \in
L^{p(.)}(\Omega ,\vartheta )\right\}
\end{equation*}%
or%
\begin{equation*}
W^{1,q(.),p(.)}\left( \Omega ,\vartheta _{0},\vartheta \right) =\left\{ f\in
L^{q(.)}(\Omega ,\vartheta _{0}):\frac{\partial f}{\partial x_{i}}\in
L^{p(.)}(\Omega ,\vartheta )\text{, }\forall i=1,2,..,d\right\}
\end{equation*}%
equipped with the norm%
\begin{equation*}
\left\Vert f\right\Vert _{W^{1,q(.),p(.)}\left( \Omega ,\vartheta
_{0},\vartheta \right) }=\left\Vert f\right\Vert _{L^{q(.)}(\Omega
,\vartheta _{0})}+\left\Vert \nabla f\right\Vert _{L^{p(.)}(\Omega
,\vartheta )}.
\end{equation*}%
Since $\vartheta \in W_{p(.)}\left( \Omega \right) $ and $\vartheta _{0}\in
W_{q(.)}\left( \Omega \right) ,$ it can be seen that $L^{p(.)}(\Omega
,\vartheta )\subset L_{loc}^{1}\left( \Omega \right) $ and $L^{q(.)}(\Omega
,\vartheta _{0})\subset L_{loc}^{1}\left( \Omega \right) .$ Therefore, the
double weighted variable exponent Sobolev spaces $W^{1,q(.),p(.)}\left(
\Omega ,\vartheta _{0},\vartheta \right) $ is well-defined.
\end{definition}

Now, we will give some basic properties of $W^{1,q(.),p(.)}\left( \Omega
,\vartheta _{0},\vartheta \right) $.

\begin{proposition}
The space $W^{1,q(.),p(.)}\left( \Omega ,\vartheta _{0},\vartheta \right) $
is a Banach space with respect to the norm $\left\Vert .\right\Vert
_{W^{1,q(.),p(.)}\left( \Omega ,\vartheta _{0},\vartheta \right) }$.
\end{proposition}

\begin{proof}
Let $\left( f_{n}\right) $ be a Cauchy sequence in $W^{1,q(.),p(.)}\left(
\Omega ,\vartheta _{0},\vartheta \right) $. Then, $\left( f_{n}\right) $ and 
$\left( \frac{\partial f_{n}}{\partial x_{i}}\right) $ are Cauchy sequences
in $L^{q(.)}(\Omega ,\vartheta _{0})$ and $L^{p(.)}(\Omega ,\vartheta )$,
respectively. Since the spaces $L^{q(.)}(\Omega ,\vartheta _{0})$ and $%
L^{p(.)}(\Omega ,\vartheta )$ are Banach spaces, the sequence $\left(
f_{n}\right) $ converges to some $f$ in $L^{q(.)}(\Omega ,\vartheta _{0})$,
and the sequence $\left( \frac{\partial f_{n}}{\partial x_{i}}\right) $
converges to some $v_{i}$ in $L^{p(.)}(\Omega ,\vartheta )$ for $i=1,2,..,d.$
Hence, we have $f,v_{i}\in L_{loc}^{1}\left( \Omega \right) $, which are
seen as distributions. Now, we will show that each $v_{i}$ coincides with $%
\frac{\partial f}{\partial x_{i}}$ in the distributional sense. For every $%
\phi \in C_{0}^{\infty }\left( \Omega \right) $, by the H\"{o}lder
inequality, we have%
\begin{eqnarray*}
\left\vert \dint\limits_{\Omega }f_{n}\phi dx-\dint\limits_{\Omega }f\phi
dx\right\vert &\leq &\left\vert \dint\limits_{\Omega }\left\vert
f_{n}-f\right\vert \left\vert \phi \right\vert dx\right\vert \\
&=&\left\vert \dint\limits_{\Omega }\left\vert f_{n}-f\right\vert \vartheta
_{0}^{\frac{1}{q(x)}}\vartheta _{0}^{-\frac{1}{q(x)}}\left\vert \phi
\right\vert dx\right\vert \\
&\leq &C\left\Vert f_{n}-f\right\Vert _{L^{q(.)}\left( \Omega ,\vartheta
_{0}\right) }\left\Vert \phi \vartheta _{0}^{-\frac{1}{q(.)}}\right\Vert
_{L^{q^{\prime }(.)}\left( \Omega \right) } \\
&\leq &C\left\Vert f_{n}-f\right\Vert _{L^{q(.)}\left( \Omega ,\vartheta
_{0}\right) }\left\Vert \phi \right\Vert _{L^{\infty }(\Omega )}\left\Vert
\vartheta _{0}^{-\frac{1}{q(.)}}\right\Vert _{L^{q^{\prime }(.)}\left( \text{%
supp}\phi \right) }
\end{eqnarray*}%
where supp$\phi \subset \Omega $ denotes the support of $\phi $ and $\frac{1%
}{q(.)}+\frac{1}{q^{\prime }(.)}=1$. Since $\vartheta _{0}\in W_{q(.)}\left(
\Omega \right) ,$ we get $\left\Vert \vartheta _{0}^{-\frac{1}{q(x)}%
}\right\Vert _{L^{q^{\prime }(.)}\left( \text{supp}\phi \right) }<\infty .$
Moreover, if we consider the fact that $f_{n}\longrightarrow f$ in $%
L^{q(.)}(\Omega ,\vartheta _{0})$, we obtain 
\begin{equation*}
\dint\limits_{\Omega }f_{n}\phi dx\longrightarrow \dint\limits_{\Omega
}f\phi dx
\end{equation*}%
as $n\longrightarrow \infty $. In similar way, using $\frac{\partial f_{n}}{%
\partial x_{i}}\longrightarrow v_{i}$ in $L^{p(.)}(\Omega ,\vartheta )$ and $%
\vartheta \in W_{p(.)}\left( \Omega \right) $, we get 
\begin{equation*}
\dint\limits_{\Omega }\frac{\partial f_{n}}{\partial x_{i}}\phi
dx\longrightarrow \dint\limits_{\Omega }v_{i}\phi dx
\end{equation*}%
as $n\longrightarrow \infty $ for all $\phi \in C_{0}^{\infty }\left( \Omega
\right) $ and $i=1,2,..,d.$ This yields 
\begin{eqnarray*}
\dint\limits_{\Omega }v_{i}\phi dx &=&\lim_{n\longrightarrow \infty
}\dint\limits_{\Omega }\frac{\partial f_{n}}{\partial x_{i}}\phi dx \\
&=&-\lim_{n\longrightarrow \infty }\dint\limits_{\Omega }f_{n}\frac{\partial
\phi }{\partial x_{i}}dx \\
&=&-\dint\limits_{\Omega }f\frac{\partial \phi }{\partial x_{i}}dx
\end{eqnarray*}%
for all $\phi \in C_{0}^{\infty }\left( \Omega \right) $ and $i=1,2,..,d.$
It follows that $v_{i}=\frac{\partial f}{\partial x_{i}}$, hence $\left(
f_{n}\right) $ converges to $f$ in $W^{1,q(.),p(.)}\left( \Omega ,\vartheta
_{0},\vartheta \right) $.
\end{proof}

\begin{remark}
The dual space of $W^{1,q(.),p(.)}\left( \Omega ,\vartheta _{0},\vartheta
\right) $ is $W^{-1,q^{\prime }(.),p^{\prime }(.)}\left( \Omega ,\vartheta
_{0}^{\ast },\vartheta ^{\ast }\right) $ where $\frac{1}{p(.)}+\frac{1}{%
p^{\prime }(.)}=1$, $\frac{1}{q(.)}+\frac{1}{q^{\prime }(.)}=1$, $\vartheta
^{\ast }=\vartheta ^{1-p^{\prime }\left( .\right) }=\vartheta ^{-\frac{1}{%
p(.)-1}}$ and $\vartheta _{0}^{\ast }=\vartheta _{0}^{1-q^{\prime }\left(
.\right) }=\vartheta _{0}^{-\frac{1}{q(.)-1}}.$
\end{remark}

It is clear that $C_{0}^{\infty }\left( \Omega \right) $ is a subspace of $%
W^{1,q(.),p(.)}\left( \Omega ,\vartheta _{0},\vartheta \right) $. Then, we
define the space $W_{0}^{1,q(.),p(.)}\left( \Omega ,\vartheta _{0},\vartheta
\right) $ as the closure of $C_{0}^{\infty }\left( \Omega \right) $ in $%
W^{1,q(.),p(.)}\left( \Omega ,\vartheta _{0},\vartheta \right) .$ Since $%
W_{0}^{1,q(.),p(.)}\left( \Omega ,\vartheta _{0},\vartheta \right) $ is a
closed subset of $W^{1,q(.),p(.)}\left( \Omega ,\vartheta _{0},\vartheta
\right) ,$ then the space $W_{0}^{1,q(.),p(.)}\left( \Omega ,\vartheta
_{0},\vartheta \right) $ is a Banach space with respect to $\left\Vert
.\right\Vert _{W^{1,q(.),p(.)}\left( \Omega ,\vartheta _{0},\vartheta
\right) }$. In particular case, if $\vartheta _{0}=\vartheta $ and $%
q(.)=p(.) $, then we have $W^{1,p(.)}\left( \Omega ,\vartheta \right)
=W^{1,p(.),p(.)}\left( \Omega ,\vartheta ,\vartheta \right) $ and $%
W_{0}^{1,p(.)}\left( \Omega ,\vartheta \right) =W_{0}^{1,p(.),p(.)}\left(
\Omega ,\vartheta ,\vartheta \right) $.

\begin{proposition}
\label{pro12}(see \cite{Gold})Let $\left\vert 
\mathbb{R}
^{d}-\Omega \right\vert >0$. If $f\in W_{0}^{1,q(.),p(.)}\left( \Omega
,\vartheta _{0},\vartheta \right) $, then the function 
\begin{equation}
\widetilde{f}(x)=\left\{ 
\begin{array}{cc}
f(x) & \text{if }x\in \Omega \\ 
0 & \text{if x}\in 
\mathbb{R}
^{d}-\Omega%
\end{array}%
\right.  \label{cano}
\end{equation}%
belongs to $W^{1,q(.),p(.)}\left( 
\mathbb{R}
^{d},\vartheta _{0},\vartheta \right) $, and for each $i=1,2,...,d,$ one has%
\begin{equation*}
\frac{\partial \widetilde{f}}{\partial x_{i}}(x)=\left\{ 
\begin{array}{cc}
\frac{\partial f}{\partial x_{i}}(x) & \text{if }x\in \Omega \\ 
0 & \text{if }x\in 
\mathbb{R}
^{d}-\Omega .%
\end{array}%
\right.
\end{equation*}
\end{proposition}

The functions in $W_{0}^{1,q(.),p(.)}\left( \Omega ,\vartheta _{0},\vartheta
\right) $ can be extended by zero outside $\Omega $ into a function in $%
W^{1,q(.),p(.)}\left( 
\mathbb{R}
^{d},\vartheta _{0},\vartheta \right) .$ Also, it is easy to see that 
\begin{equation*}
W_{0}^{1,q(.),p(.)}\left( 
\mathbb{R}
^{d},\vartheta _{0},\vartheta \right) =W^{1,q(.),p(.)}\left( 
\mathbb{R}
^{d},\vartheta _{0},\vartheta \right) .
\end{equation*}%
Moreover, by Proposition \ref{pro12}, we can write that%
\begin{equation*}
W_{0}^{1,q(.),p(.)}\left( \Omega ,\vartheta _{0},\vartheta \right) =\left\{
f\in L^{q(.)}(\Omega ,\vartheta _{0}):\left\vert \nabla f\right\vert
^{p(.)}\in L^{1}(\Omega ,\vartheta ),\text{ }f=0\text{ on }\partial \Omega
\right\}
\end{equation*}%
equipped with the norm%
\begin{equation*}
\left\Vert f\right\Vert _{W_{0}^{1,q(.),p(.)}\left( \Omega ,\vartheta
_{0},\vartheta \right) }=\left\Vert f\right\Vert _{L^{q(.)}(\Omega
,\vartheta _{0})}+\left\Vert \nabla f\right\Vert _{L^{p(.)}(\Omega
,\vartheta )}.
\end{equation*}

The importance of $W_{0}^{1,q(.),p(.)}\left( \Omega ,\vartheta
_{0},\vartheta \right) $ can be seen in various applications, such as
Dirichlet problem for elliptic partial differential equations. That means
the zero extension property above allows us to consider the space as a
solution space for problems with Dirichlet type boundary conditions.

Now, we consider the Poincar\'{e} inequality in the space $%
W_{0}^{1,q(.),p(.)}\left( \Omega ,\vartheta _{0},\vartheta \right) .$ Let $%
A\subset 
\mathbb{R}
^{d}.$ We define 
\begin{equation*}
p_{A}^{-}=\underset{x\in A\cap \Omega }{\text{essinf}}p(x)\text{, \ \ \ \ \
\ }p_{A}^{+}=\underset{x\in A\cap \Omega }{\text{esssup}}p(x).
\end{equation*}%
If $p_{\Omega }^{+}<\infty $ and if there exists $r>0$ such that every $x\in
\Omega $ either%
\begin{equation*}
p_{B\left( x,r\right) }^{-}\geq d
\end{equation*}%
or%
\begin{equation*}
p_{B\left( x,r\right) }^{+}\leq \frac{dp_{B\left( x,r\right) }^{-}}{%
d-p_{B\left( x,r\right) }^{-}}
\end{equation*}%
is valid, then the variable exponent $p\left( .\right) $ is said to
satisfies the jump condition in $\Omega $ with constant $r.$ Moreover we put%
\begin{equation*}
p_{B\left( x,r\right) }^{\ast }=\left\{ 
\begin{array}{c}
\frac{dp_{B\left( x,r\right) }^{-}}{d-p_{B\left( x,r\right) }^{-}},\text{ \
\ \ \ }p_{B\left( x,r\right) }^{-}<d \\ 
p_{B\left( x,r\right) }^{+},\text{ \ \ \ \ }p_{B\left( x,r\right) }^{-}\geq d%
\end{array}%
\right. 
\end{equation*}%
It is clear that if $\Omega $ is bounded and if $p\left( .\right) $ is
continuous in $\overline{\Omega }$, then $p\left( .\right) $ satisfies the
jump condition in $\Omega $ with some $r>0,$ see \cite{Har4}, \cite{Unal2}.

\begin{remark}
\label{Remark}Let $\Omega \subset 
\mathbb{R}
^{d}$ be a bounded set. Then, the claim of Proposition 2.4 in \cite{Liu}
satisfies even if $p\left( .\right) =1.$ This yields that the space $%
L_{\vartheta }^{p\left( .\right) }\left( \Omega \right) $ is continuously
embedded in $L_{\vartheta }^{1}\left( \Omega \right) $.
\end{remark}

Now, we are ready to consider the Poincar\'{e} inequality for $%
W_{0}^{1,q(.),p(.)}\left( \Omega ,\vartheta _{0},\vartheta \right) .$

\begin{theorem}
\label{poincare}Let $\Omega $ be a bounded open domain in $%
\mathbb{R}
^{d}$ with smooth boundary $\partial \Omega .$ Moreover, assume that the
exponent $q\left( .\right) $ holds the jump condition in $\Omega $ with
constant $r>0$ and $\vartheta _{0}\prec \vartheta $. Then, there is a
constant $C>0$ such that 
\begin{equation}
\left\Vert f\right\Vert _{L^{q(.)}(\Omega ,\vartheta _{0})}\leq C\left\Vert
\nabla f\right\Vert _{L^{p(.)}(\Omega ,\vartheta )}  \label{3}
\end{equation}%
for every $f\in W_{0}^{1,q(.),p(.)}\left( \Omega ,\vartheta _{0},\vartheta
\right) $.
\end{theorem}

\begin{proof}
Since $\Omega $ is a bounded set, $\overline{\Omega }$ is compact. Then, we
can find $x_{1},x_{2},...,x_{n}$ such that%
\begin{equation*}
\Omega \subset \tbigcup\limits_{n=1}^{t}B\left( x_{n},r\right) .
\end{equation*}%
Because of the fact that $f\in W_{0}^{1,q(.),p(.)}\left( \Omega ,\vartheta
_{0},\vartheta \right) ,$ the function $\widetilde{f}$ can be taken as (\ref%
{cano}). Since the exponent $q\left( .\right) $ holds the jump condition, we
get by \cite[Proposition 2.4]{Liu} that%
\begin{eqnarray}
\left\Vert f\right\Vert _{L^{q(.)}\left( \Omega ,\vartheta _{0}\right) }
&=&\left\Vert \widetilde{f}\right\Vert _{L^{q(.)}\left( 
\mathbb{R}
^{d},\vartheta _{0}\right) }\leq \left\Vert \widetilde{f}\left[ \chi
_{B\left( x_{1},r\right) }+...+\chi _{B\left( x_{n},r\right) }\right]
\right\Vert _{L^{q(.)}\left( 
\mathbb{R}
^{d},\vartheta _{0}\right) }  \notag \\
&\leq &\tsum\limits_{n=1}^{t}\left\Vert \widetilde{f}\right\Vert
_{L^{q(.)}\left( B\left( x_{n},r\right) ,\vartheta _{0}\right) }\leq
c\tsum\limits_{n=1}^{t}\left\Vert \widetilde{f}\right\Vert _{L^{q_{B\left(
x_{n},r\right) }^{\ast }}\left( B\left( x_{n},r\right) ,\vartheta
_{0}\right) }  \notag \\
&\leq &c\tsum\limits_{n=1}^{t}\left( \left\Vert \widetilde{f}-f_{B\left(
x_{n},r\right) }^{\ast }\right\Vert _{L^{q_{B\left( x_{n},r\right) }^{\ast
}}\left( B\left( x_{n},r\right) ,\vartheta _{0}\right) }\right.   \notag \\
&&\left. +\left\vert f_{B\left( x_{n},r\right) }^{\ast }\right\vert
\left\Vert 1\right\Vert _{L^{q_{B\left( x_{n},r\right) }^{\ast }}\left(
B\left( x_{n},r\right) ,\vartheta _{0}\right) }\right) .  \label{4}
\end{eqnarray}%
It is note that the function $f_{B\left( x_{n},r\right) }^{\ast }$ is
average of $\widetilde{f}$ over the balls $B\left( x_{n},r\right) $ and
defined as $f_{B\left( x_{n},r\right) }^{\ast }=\frac{1}{\left\vert B\left(
x_{n},r\right) \right\vert }\tint\limits_{B\left( x_{n},r\right) }\widetilde{%
f}\left( x\right) \vartheta \left( x\right) dx,$ see \cite{Hei}. It is clear
that $q_{B\left( x_{n},r\right) }^{-}\leq q\left( .\right) .$ Moreover, if
we use the Poincar\'{e} inequality over the balls (see \cite[Section 1]{Hei}%
) and the embedding $L^{q\left( .\right) }\left( B\left( x_{n},r\right)
,\vartheta _{0}\right) \hookrightarrow L^{q_{B\left( x_{n},r\right)
}^{-}}\left( B\left( x_{n},r\right) ,\vartheta _{0}\right) $ (see \cite[%
Proposition 2.4]{Liu}), then we obtain%
\begin{eqnarray*}
\left\Vert \widetilde{f}-f_{B\left( x_{n},r\right) }^{\ast }\right\Vert
_{L^{q_{B\left( x_{n},r\right) }^{\ast }}\left( B\left( x_{n},r\right)
,\vartheta _{0}\right) } &\leq &rc\left\Vert \nabla \widetilde{f}\right\Vert
_{L^{q_{B\left( x_{n},r\right) }^{-}}\left( B\left( x_{n},r\right)
,\vartheta _{0}\right) } \\
&\leq &rc\left\Vert \nabla \widetilde{f}\right\Vert _{L^{q\left( .\right)
}\left( B\left( x_{n},r\right) ,\vartheta _{0}\right) } \\
&\leq &rc\left\Vert \nabla f\right\Vert _{L^{q\left( .\right) }\left( \Omega
,\vartheta _{0}\right) }
\end{eqnarray*}%
for all $n=1,2,...,t.$ Since $q\left( .\right) <p\left( .\right) $ and $%
\vartheta _{0}\prec \vartheta ,$ we have $L^{p\left( .\right) }\left( \Omega
,\vartheta \right) \hookrightarrow L^{p\left( .\right) }\left( \Omega
,\vartheta _{0}\right) \hookrightarrow L^{q\left( .\right) }\left( \Omega
,\vartheta _{0}\right) .$ This follows that%
\begin{equation*}
\left\Vert \widetilde{f}-f_{B\left( x_{n},r\right) }^{\ast }\right\Vert
_{L^{q_{B\left( x_{n},r\right) }^{\ast }}\left( B\left( x_{n},r\right)
,\vartheta _{0}\right) }\leq rc^{\ast \ast }\left\Vert \nabla f\right\Vert
_{L^{p\left( .\right) }\left( \Omega ,\vartheta \right) }.
\end{equation*}

Moreover, if we use the Poincar\'{e} inequality in $L_{\vartheta }^{1}\left(
\Omega \right) $ and Remark \ref{Remark}, then we get%
\begin{eqnarray*}
\left\vert f_{B\left( x_{n},r\right) }^{\ast }\right\vert &\leq &\frac{C}{%
r^{d}}\tint\limits_{\Omega }\left\vert f\left( x\right) \right\vert
\vartheta \left( x\right) dx\leq \frac{C}{r^{d}}diam\left( \Omega \right)
\tint\limits_{\Omega }\left\vert \nabla f\left( x\right) \right\vert
\vartheta \left( x\right) dx \\
&\leq &\frac{C}{r^{d}}diam\left( \Omega \right) c^{\ast \ast \ast
}\left\Vert \nabla f\right\Vert _{L^{q\left( .\right) }\left( \Omega
,\vartheta _{0}\right) } \\
&\leq &\frac{C}{r^{d}}diam\left( \Omega \right) c^{\ast \ast \ast
}\left\Vert \nabla f\right\Vert _{L^{p\left( .\right) }\left( \Omega
,\vartheta \right) }
\end{eqnarray*}%
for all $n=1,2,...,t.$ Since $\vartheta \in W_{p(.)}\left( \Omega \right) ,$
we have%
\begin{equation*}
\rho _{q_{B\left( x_{n},r\right) }^{\ast },\vartheta _{0}}\left( \chi
_{B\left( x_{n},r\right) }\right) =\tint\limits_{B\left( x_{n},r\right)
}\vartheta \left( x\right) dx<\infty .
\end{equation*}%
This yields that $\left\Vert 1\right\Vert _{L^{q_{B\left( x_{n},r\right)
}^{\ast }}\left( B\left( x_{n},r\right) ,\vartheta _{0}\right) }$ depends
only on $q_{B\left( x_{n},r\right) }^{\ast }$. Hence the claim follows from
the inequality (\ref{4}).
\end{proof}

The inequality (\ref{3}) is well known for the classical weighted Sobolev
spaces $W_{0}^{1,p}\left( \Omega ,\vartheta \right) $ under some conditions,
see \cite{Maz}. From now on, we assume that necessary conditions satisfy the
inequality (\ref{3}) in $W_{0}^{1,q(.),p(.)}\left( \Omega ,\vartheta
_{0},\vartheta \right) $.

\begin{definition}
\label{yeninorm}By the Theorem \ref{poincare}, we can present the norm on $%
W_{0}^{1,q(.),p(.)}\left( \Omega ,\vartheta _{0},\vartheta \right) $ denoted
by 
\begin{equation*}
\left\Vert \left\vert f\right\vert \right\Vert _{W_{0}^{1,q(.),p(.)}\left(
\Omega ,\vartheta _{0},\vartheta \right) }=\left\Vert \nabla f\right\Vert
_{L^{p(.)}(\Omega ,\vartheta )}
\end{equation*}%
for every $f\in W_{0}^{1,q(.),p(.)}\left( \Omega ,\vartheta _{0},\vartheta
\right) .$ It is note that the norms $\left\Vert .\right\Vert
_{W_{0}^{1,q(.),p(.)}\left( \Omega ,\vartheta _{0},\vartheta \right) }$ and $%
\left\Vert \left\vert .\right\vert \right\Vert _{W_{0}^{1,q(.),p(.)}\left(
\Omega ,\vartheta _{0},\vartheta \right) }$ are equivalent norms on $%
W_{0}^{1,q(.),p(.)}\left( \Omega ,\vartheta _{0},\vartheta \right) .$ Then,
the space $W_{0}^{1,q(.),p(.)}\left( \Omega ,\vartheta _{0},\vartheta
\right) $ is continuously embedded in $L^{q(.)}(\Omega ,\vartheta _{0})$ if
and only if the inequality (\ref{3}) is satisfied for all $f\in
W_{0}^{1,q(.),p(.)}\left( \Omega ,\vartheta _{0},\vartheta \right) .$
\end{definition}

\section{Application}

In this section, we discuss the $p\left( .\right) $-Laplace operator $%
-\Delta _{p(.),\vartheta }=-\func{div}\left( \vartheta (x)\left\vert \nabla
f\right\vert ^{p(x)-2}\nabla f\right) .$ Let us consider the functional%
\begin{equation*}
J\left( f\right) =\dint\limits_{\Omega }\left( \frac{1}{p\left( x\right) }%
\left\vert \nabla f\right\vert ^{p\left( x\right) }\vartheta \left( x\right)
-\frac{1}{q\left( x\right) }\left\vert f\right\vert ^{q\left( x\right)
}\vartheta _{0}\left( x\right) \right) dx
\end{equation*}%
for all $f\in W_{0}^{1,q(.),p(.)}\left( \Omega ,\vartheta _{0},\vartheta
\right) .$ Then $J\in C^{1}\left( W_{0}^{1,q(.),p(.)}\left( \Omega
,\vartheta _{0},\vartheta \right) ,%
\mathbb{R}
\right) $, and the $p\left( .\right) $-Laplace operator is the derivative
operator of $J$ in the weak sense satisfies%
\begin{equation*}
\left\langle J^{\prime }\left( f\right) ,g\right\rangle
=\dint\limits_{\Omega }\left( \vartheta \left( x\right) \left\vert \nabla
f\right\vert ^{p\left( x\right) -2}\nabla f\nabla g-\vartheta _{0}\left(
x\right) \left\vert f\right\vert ^{q\left( x\right) -2}fg\right) dx
\end{equation*}%
for all $f,g\in W_{0}^{1,q(.),p(.)}\left( \Omega ,\vartheta _{0},\vartheta
\right) .$

\begin{definition}
We call that $f\in W_{0}^{1,q(.),p(.)}\left( \Omega ,\vartheta
_{0},\vartheta \right) $ is a weak solution of problem (\ref{P1}) if%
\begin{equation*}
\tint\limits_{\Omega }\vartheta (x)\left\vert \nabla f\right\vert
^{p(x)-2}\nabla f.\nabla gdx=\tint\limits_{\Omega }\vartheta
_{0}(x)\left\vert f\right\vert ^{q(x)-2}fgdx
\end{equation*}%
for all $g\in W_{0}^{1,q(.),p(.)}\left( \Omega ,\vartheta _{0},\vartheta
\right) .$
\end{definition}

Let $p\left( .\right) \in P\left( \Omega \right) .$ Now, we define Sobolev
conjugate of $p\left( .\right) $ as%
\begin{equation*}
p^{\ast }\left( .\right) =\left\{ 
\begin{array}{c}
\frac{dp\left( .\right) }{d-p\left( .\right) },\text{ \ \ \ \ }p\left(
.\right) <d \\ 
\infty ,\text{ \ \ \ \ \ }p\left( .\right) \geq d%
\end{array}%
\right. .
\end{equation*}

\begin{theorem}
\label{compactem2}(see \cite{Dien1}, \cite{Kim})Suppose that $\Omega \subset 
\mathbb{R}
^{d}$ is an open, bounded set with Lipschitz boundary and $p\left( .\right)
\in C^{+}\left( \overline{\Omega }\right) ,$ $p\left( .\right) \in P^{\log
}\left( \Omega \right) $ with $1<p^{-}\leq p^{+}<d.$ If $r\left( .\right)
\in L^{\infty }\left( \Omega \right) $ with $r^{-}>1$ satisfies $r\left(
x\right) \leq p^{\ast }\left( x\right) $ for every $x\in \Omega ,$ then we
obtain the embedding $W^{1,p\left( .\right) }\left( \Omega \right)
\hookrightarrow L^{r\left( .\right) }\left( \Omega \right) .$ Moreover, the
compact embedding $W^{1,p\left( .\right) }\left( \Omega \right)
\hookrightarrow \hookrightarrow L^{r\left( .\right) }\left( \Omega \right) $
holds if $\underset{x\in \Omega }{\text{inf}}\left( p^{\ast }\left( x\right)
-r\left( x\right) \right) >0.$
\end{theorem}

\begin{theorem}
\label{comson}Suppose that $p\left( .\right) ,q\left( .\right) \in C\left( 
\overline{\Omega }\right) $ and $1<p\left( x\right) ,q\left( x\right) $ for
all $x\in \overline{\Omega }$ and moreover,
\end{theorem}

\begin{enumerate}
\item[\textit{(I)}] $0<\vartheta _{1}\in L^{\alpha \left( .\right) }\left(
\Omega \right) $\textit{\ with }$1<\alpha \left( .\right) \in C\left( 
\overline{\Omega }\right) $,

\item[\textit{(II)}] $\vartheta _{0}^{-\frac{t\left( .\right) }{q\left(
.\right) -t\left( .\right) }},\vartheta ^{-\frac{t\left( .\right) }{p\left(
.\right) -t\left( .\right) }}\in L^{1}\left( \Omega \right) $\textit{\ where 
}$t\left( .\right) \in C\left( \overline{\Omega }\right) $\textit{\ and }$%
1<t\left( .\right) <q\left( .\right) <p\left( .\right) .$

\item[\textit{(III)}] $\vartheta _{0}\left( x\right) \geq c>0$\textit{\ for
all }$x\in \Omega .$
\end{enumerate}

\textit{Then we get the compact embedding }$W_{0}^{1,q(.),p(.)}\left( \Omega
,\vartheta _{0},\vartheta \right) \hookrightarrow \hookrightarrow L^{r\left(
.\right) }\left( \Omega ,\vartheta _{1}\right) $\textit{\ for every }$%
r\left( .\right) \in C\left( \overline{\Omega }\right) $\textit{\ and }$%
1<r\left( .\right) <\frac{t^{\ast }\left( .\right) }{\beta \left( .\right) }$%
\textit{\ where }$\frac{1}{\alpha \left( .\right) }+\frac{1}{\beta \left(
.\right) }=1$\textit{.}

\begin{proof}
Let $f\in W_{0}^{1,q(.),p(.)}\left( \Omega ,\vartheta _{0},\vartheta \right)
.$ Then we write that $f\in L^{q\left( .\right) }\left( \Omega ,\vartheta
_{0}\right) $ and $\nabla f\in L^{p\left( .\right) }\left( \Omega ,\vartheta
\right) $. This follows \textit{(II) }that $\rho _{L^{\frac{p\left( .\right) 
}{t\left( .\right) }}\left( \Omega \right) }\left( \left\vert \nabla
f\right\vert ^{t\left( .\right) }\vartheta ^{\frac{t\left( .\right) }{%
p\left( .\right) }}\right) <\infty $ and $\rho _{L^{\frac{p\left( .\right) }{%
p\left( .\right) -t\left( .\right) }}\left( \Omega \right) }\left( \vartheta
^{-\frac{t\left( .\right) }{p\left( .\right) }}\right) <\infty .$ By the H%
\"{o}lder inequality, we have%
\begin{equation*}
\tint\limits_{\Omega }\left\vert \nabla f\left( x\right) \right\vert
^{t\left( x\right) }dx\leq c_{h}\left\Vert \left\vert \nabla f\right\vert
^{t\left( .\right) }\vartheta ^{\frac{t\left( .\right) }{p\left( .\right) }%
}\right\Vert _{L^{\frac{p\left( .\right) }{t\left( .\right) }}\left( \Omega
\right) }\left\Vert \vartheta ^{-\frac{t\left( .\right) }{p\left( .\right) }%
}\right\Vert _{L^{\frac{p\left( .\right) }{p\left( .\right) -t\left(
.\right) }}\left( \Omega \right) }.
\end{equation*}

If we consider the \cite[Proposition 2.4]{Liu} and \textit{(II)}, then we
get $\left\Vert \vartheta ^{-\frac{t\left( .\right) }{p\left( .\right) }%
}\right\Vert _{L^{\frac{p\left( .\right) }{p\left( .\right) -t\left(
.\right) }}\left( \Omega \right) }^{\frac{p^{-}}{p^{+}-t^{-}}}<\infty $ and $%
\left\Vert \vartheta ^{-\frac{t\left( .\right) }{p\left( .\right) }%
}\right\Vert _{L^{\frac{p\left( .\right) }{p\left( .\right) -t\left(
.\right) }}\left( \Omega \right) }^{\frac{p^{+}}{p^{-}-t^{+}}}<\infty .$
This follows that%
\begin{equation*}
\left\Vert \vartheta ^{-\frac{t\left( .\right) }{p\left( .\right) }%
}\right\Vert _{L^{\frac{p\left( .\right) }{p\left( .\right) -t\left(
.\right) }}\left( \Omega \right) }\leq \left( \tint\limits_{\Omega }\left(
\vartheta \left( x\right) \right) ^{-\frac{t\left( x\right) }{p\left(
x\right) -t\left( x\right) }}dx+1\right) ^{\frac{p^{+}-t^{-}}{p^{-}}}\leq
c_{1}
\end{equation*}%
and 
\begin{equation}
\tint\limits_{\Omega }\left\vert \nabla f\left( x\right) \right\vert
^{t\left( x\right) }dx\leq c_{h}c_{1}\left\Vert \left\vert \nabla
f\right\vert ^{t\left( .\right) }\vartheta ^{\frac{t\left( .\right) }{%
p\left( .\right) }}\right\Vert _{L^{\frac{p\left( .\right) }{t\left(
.\right) }}\left( \Omega \right) }.  \label{komhol}
\end{equation}%
In general, we can suppose that $\tint\limits_{\Omega }\left\vert \nabla
f\left( x\right) \right\vert ^{t\left( x\right) }dx>1.$ By \cite[Proposition
2.4]{Liu} and (\ref{komhol}) when $\tint\limits_{\Omega }\left\vert \nabla
f\left( x\right) \right\vert ^{p\left( x\right) }\vartheta \left( x\right)
dx\leq 1$, we have%
\begin{eqnarray*}
\left\Vert \nabla f\right\Vert _{L^{t\left( .\right) }\left( \Omega \right)
}^{t^{-}} &\leq &c_{h}c_{1}\left\Vert \left\vert \nabla f\right\vert
^{t\left( .\right) }\vartheta ^{\frac{t\left( .\right) }{p\left( .\right) }%
}\right\Vert _{L^{\frac{p\left( .\right) }{t\left( .\right) }}\left( \Omega
\right) } \\
&\leq &c_{h}c_{1}\left( \tint\limits_{\Omega }\left\vert \nabla f\left(
x\right) \right\vert ^{p\left( x\right) }\vartheta \left( x\right) dx\right)
^{\frac{t^{-}}{p^{+}}}\leq c_{h}c_{1}\left\Vert \nabla f\right\Vert
_{L^{p\left( .\right) }\left( \Omega ,\vartheta \right) }^{\frac{p^{-}t^{-}}{%
p^{+}}}.
\end{eqnarray*}%
That means%
\begin{equation}
\left\Vert \nabla f\right\Vert _{L^{t\left( .\right) }\left( \Omega \right)
}\leq C\left\Vert \nabla f\right\Vert _{L^{p\left( .\right) }\left( \Omega
,\vartheta \right) }^{\frac{p^{-}}{p^{+}}}  \label{komhol1}
\end{equation}%
where $C=\left( c_{h}c_{1}\right) ^{\frac{1}{t^{-}}}>0$. By similar method,
if $\tint\limits_{\Omega }\left\vert \nabla f\left( x\right) \right\vert
^{p\left( x\right) }\vartheta \left( x\right) dx>1,$ we obtain%
\begin{equation}
\left\Vert \nabla f\right\Vert _{L^{t\left( .\right) }\left( \Omega \right)
}\leq C\left\Vert \nabla f\right\Vert _{L^{p\left( .\right) }\left( \Omega
,\vartheta \right) }^{\frac{p^{+}t^{+}}{p^{-}t^{-}}}  \label{komhol2}
\end{equation}%
where $C=\left( c_{h}c_{1}\right) ^{\frac{1}{t^{-}}}>0$. If we consider the
inequalities (\ref{komhol1}) and (\ref{komhol2}), then we have $\nabla f\in
L^{t\left( .\right) }\left( \Omega \right) .$ In addition, by $t\left(
.\right) <q\left( .\right) ,$ $\left\vert \Omega \right\vert <\infty $, 
\textit{(II)} and \textit{(III)}, we get that $L^{q\left( .\right) }\left(
\Omega ,\vartheta _{0}\right) \hookrightarrow L^{t\left( .\right) }\left(
\Omega ,\vartheta _{0}\right) \hookrightarrow L^{t\left( .\right) }\left(
\Omega \right) $. Therefore, we have $f\in L^{t\left( .\right) }\left(
\Omega \right) $. This follows that $f\in W^{1,t\left( .\right) }\left(
\Omega \right) .$ Hence, the inclusion $f\in W_{0}^{1,q(.),p(.)}\left(
\Omega ,\vartheta _{0},\vartheta \right) \subset W^{1,t\left( .\right)
}\left( \Omega \right) $ is satisfied. Using the Banach Theorem in \cite{Car}%
, we get%
\begin{equation}
W_{0}^{1,q(.),p(.)}\left( \Omega ,\vartheta _{0},\vartheta \right)
\hookrightarrow W^{1,t\left( .\right) }\left( \Omega \right) .
\label{surgom}
\end{equation}%
By Theorem \ref{compactem2}, we have compact embedding%
\begin{equation}
W^{1,t\left( .\right) }\left( \Omega \right) \hookrightarrow \hookrightarrow
L^{s\left( .\right) }\left( \Omega \right)  \label{comgom}
\end{equation}%
for $s\left( .\right) <t^{\ast }\left( .\right) $. Now, we define $s\left(
.\right) =r\left( .\right) \beta \left( .\right) $. By the H\"{o}lder
inequality for variable exponent Lebesgue space and \textit{(I)}, we have%
\begin{equation*}
\tint\limits_{\Omega }\left\vert f\left( x\right) \right\vert ^{r\left(
x\right) }\vartheta _{1}\left( x\right) dx\leq c_{h}\left\Vert \left\vert
f\right\vert ^{r\left( .\right) }\right\Vert _{L^{\beta \left( .\right)
}\left( \Omega \right) }\left\Vert \vartheta _{1}\right\Vert _{L^{\alpha
\left( .\right) }\left( \Omega \right) }<\infty \text{.}
\end{equation*}%
This follows that $W_{0}^{1,q(.),p(.)}\left( \Omega ,\vartheta
_{0},\vartheta \right) \subset L^{r\left( .\right) }\left( \Omega ,\vartheta
_{1}\right) $. If we consider the Banach Theorem in \cite{Car}, then we get $%
W_{0}^{1,q(.),p(.)}\left( \Omega ,\vartheta _{0},\vartheta \right)
\hookrightarrow L^{r\left( .\right) }\left( \Omega ,\vartheta _{1}\right) $.
Now, we take a sequence $\left( f_{n}\right) _{n\in 
\mathbb{N}
}\subset W_{0}^{1,q(.),p(.)}\left( \Omega ,\vartheta _{0},\vartheta \right) $
such that $f_{n}\rightharpoonup 0$ in $W_{0}^{1,q(.),p(.)}\left( \Omega
,\vartheta _{0},\vartheta \right) $ as $n\longrightarrow \infty $. This
follows that $f_{n}\rightharpoonup 0$ in $W^{1,t\left( .\right) }\left(
\Omega \right) $ by (\ref{surgom}). Moreover, if we consider (\ref{comgom}),
then we get that $f_{n}\longrightarrow 0$ in $L^{s\left( .\right) }\left(
\Omega \right) $. Hence, we have%
\begin{equation*}
\tint\limits_{\Omega }\left\vert f_{n}\left( x\right) \right\vert ^{r\left(
x\right) }\vartheta _{1}\left( x\right) dx\leq c_{h}\left\Vert \left\vert
f_{n}\right\vert ^{r\left( .\right) }\right\Vert _{L^{\beta \left( .\right)
}\left( \Omega \right) }\left\Vert \vartheta _{1}\right\Vert _{L^{\alpha
\left( .\right) }\left( \Omega \right) }\longrightarrow 0
\end{equation*}%
that is, $f_{n}\longrightarrow 0$ in $L^{r\left( .\right) }\left( \Omega
,\vartheta _{1}\right) $. This completes the proof.
\end{proof}

\begin{corollary}
Assume that all assumptions of Theorem \ref{comson} are satisfied. Then
there exist $C_{1},C_{2}>0$ such that%
\begin{equation*}
\tint\limits_{\Omega }\left\vert f\left( x\right) \right\vert ^{r\left(
x\right) }\vartheta _{1}\left( x\right) dx\leq \left\{ 
\begin{array}{c}
C_{1}\left( \left\Vert f\right\Vert _{W_{0}^{1,q(.),p(.)}\left( \Omega
,\vartheta _{0},\vartheta \right) }\right) ^{r^{+}}\text{, \ \ \ \ if }%
\left\Vert f\right\Vert _{W_{0}^{1,q(.),p(.)}\left( \Omega ,\vartheta
_{0},\vartheta \right) }>1 \\ 
C_{2}\left( \left\Vert f\right\Vert _{W_{0}^{1,q(.),p(.)}\left( \Omega
,\vartheta _{0},\vartheta \right) }\right) ^{r^{-}}\text{, \ \ \ \ if }%
\left\Vert f\right\Vert _{W_{0}^{1,q(.),p(.)}\left( \Omega ,\vartheta
_{0},\vartheta \right) }<1%
\end{array}%
\right.
\end{equation*}%
for all $f\in W_{0}^{1,q(.),p(.)}\left( \Omega ,\vartheta _{0},\vartheta
\right) $.
\end{corollary}

\begin{proof}
If we consider the Theorem \ref{comson}, then we have $W_{0}^{1,q(.),p(.)}%
\left( \Omega ,\vartheta _{0},\vartheta \right) \hookrightarrow
\hookrightarrow L^{r^{+}}\left( \Omega ,\vartheta _{1}\right) $ and $%
W_{0}^{1,q(.),p(.)}\left( \Omega ,\vartheta _{0},\vartheta \right)
\hookrightarrow \hookrightarrow L^{r^{-}}\left( \Omega ,\vartheta
_{1}\right) $ for $1<r^{-}\leq r\left( .\right) \leq r^{+}<\frac{t^{\ast
}\left( .\right) }{\beta \left( .\right) }$. Therefore, there are $%
c_{1},c_{2}>0$ such that%
\begin{equation*}
\left\Vert f\right\Vert _{L^{r^{+}}\left( \Omega ,\vartheta _{1}\right)
}=\left( \tint\limits_{\Omega }\left\vert f\left( x\right) \right\vert
^{r^{+}}\vartheta _{1}\left( x\right) dx\right) ^{\frac{1}{r^{+}}}\leq
c_{1}\left\Vert f\right\Vert _{W_{0}^{1,q(.),p(.)}\left( \Omega ,\vartheta
_{0},\vartheta \right) }
\end{equation*}%
and%
\begin{equation*}
\left\Vert f\right\Vert _{L^{r^{-}}\left( \Omega ,\vartheta _{1}\right)
}=\left( \tint\limits_{\Omega }\left\vert f\left( x\right) \right\vert
^{r^{-}}\vartheta _{1}\left( x\right) dx\right) ^{\frac{1}{r^{-}}}\leq
c_{2}\left\Vert f\right\Vert _{W_{0}^{1,q(.),p(.)}\left( \Omega ,\vartheta
_{0},\vartheta \right) }
\end{equation*}%
for all $f\in W_{0}^{1,q(.),p(.)}\left( \Omega ,\vartheta _{0},\vartheta
\right) $. This implies that%
\begin{eqnarray*}
\tint\limits_{\Omega }\left\vert f\left( x\right) \right\vert ^{r\left(
x\right) }\vartheta _{1}\left( x\right) dx &\leq &\tint\limits_{\Omega
}\left( \left\vert f\left( x\right) \right\vert ^{r^{+}}+\left\vert f\left(
x\right) \right\vert ^{r^{-}}\right) \vartheta _{1}\left( x\right) dx \\
&\leq &c_{1}^{r^{+}}\left( \left\Vert f\right\Vert
_{W_{0}^{1,q(.),p(.)}\left( \Omega ,\vartheta _{0},\vartheta \right)
}\right) ^{r^{+}}+c_{2}^{r^{-}}\left( \left\Vert f\right\Vert
_{W_{0}^{1,q(.),p(.)}\left( \Omega ,\vartheta _{0},\vartheta \right)
}\right) ^{r^{-}} \\
&\leq &\left\{ 
\begin{array}{c}
C_{1}\left( \left\Vert f\right\Vert _{W_{0}^{1,q(.),p(.)}\left( \Omega
,\vartheta _{0},\vartheta \right) }\right) ^{r^{+}}\text{, \ \ \ \ if }%
\left\Vert f\right\Vert _{W_{0}^{1,q(.),p(.)}\left( \Omega ,\vartheta
_{0},\vartheta \right) }>1 \\ 
C_{2}\left( \left\Vert f\right\Vert _{W_{0}^{1,q(.),p(.)}\left( \Omega
,\vartheta _{0},\vartheta \right) }\right) ^{r^{-}}\text{, \ \ \ \ if }%
\left\Vert f\right\Vert _{W_{0}^{1,q(.),p(.)}\left( \Omega ,\vartheta
_{0},\vartheta \right) }<1%
\end{array}%
\right. \text{.}
\end{eqnarray*}
\end{proof}

From now on, we assume that $\vartheta _{0}$ and $\vartheta $ satisfy 
\textit{(I)} and \textit{(II)}, \textit{(III)}, respectively. By the similar
method in \cite[Theorem 3.1]{FanZ}, the following theorem is easy to see.

\begin{theorem}
\label{pro6}Assume that $p^{\prime }\left( .\right) $ and $q^{\prime }\left(
.\right) $ are the conjugate exponents of $p\left( .\right) $ and $q\left(
.\right) $, respectively. Moreover, let $\vartheta ^{\ast }=\vartheta
^{1-p\left( .\right) }$ and $\vartheta _{0}^{\ast }=\vartheta
_{0}^{1-q\left( .\right) }.$ Then, we have

\begin{enumerate}
\item[\textit{(i)}] $J^{\prime }:W_{0}^{1,q(.),p(.)}\left( \Omega ,\vartheta
_{0},\vartheta \right) \longrightarrow W_{0}^{-1,q^{\prime }(.),p^{\prime
}(.)}\left( \Omega ,\vartheta _{0}^{\ast },\vartheta ^{\ast }\right) $ is
continuous, bounded and strictly monotone operator.

\item[\textit{(ii)}] $J^{\prime }$ is a mapping of type $\left( S_{+}\right)
,$ i.e., if $f_{n}\rightharpoonup f$ in $W_{0}^{1,q(.),p(.)}\left( \Omega
,\vartheta _{0},\vartheta \right) $ and $\limsup_{n\longrightarrow \infty
}\left\langle J^{\prime }\left( f_{n}\right) -J^{\prime }\left( f\right)
,f_{n}-f\right\rangle \leq 0$, then $f_{n}\longrightarrow f$ in $%
W_{0}^{1,q(.),p(.)}\left( \Omega ,\vartheta _{0},\vartheta \right) .$

\item[\textit{(iii)}] $J^{\prime }:W_{0}^{1,q(.),p(.)}\left( \Omega
,\vartheta _{0},\vartheta \right) \longrightarrow W_{0}^{-1,q^{\prime
}(.),p^{\prime }(.)}\left( \Omega ,\vartheta _{0}^{\ast },\vartheta ^{\ast
}\right) $ is a homeomorphism.
\end{enumerate}
\end{theorem}

\begin{theorem}
\label{coercive}The energy functional $J$ is coercive and bounded below.
\end{theorem}

\begin{proof}
Let $f\in W_{0}^{1,q(.),p(.)}\left( \Omega ,\vartheta _{0},\vartheta \right) 
$ and $\left\Vert f\right\Vert _{W_{0}^{1,q(.),p(.)}\left( \Omega ,\vartheta
_{0},\vartheta \right) }>1.$ If we consider the Definition \ref{yeninorm}
and \cite[Proposition 2.4]{Liu} (or \cite{A2}), then we have%
\begin{eqnarray*}
J\left( f\right) &=&\dint\limits_{\Omega }\frac{1}{p\left( x\right) }%
\left\vert \nabla f\right\vert ^{p\left( x\right) }\vartheta \left( x\right)
dx-\dint\limits_{\Omega }\frac{1}{q\left( x\right) }\left\vert f\right\vert
^{q\left( x\right) }\vartheta _{0}\left( x\right) dx \\
&\geq &\frac{1}{p^{+}}\dint\limits_{\Omega }\left\vert \nabla f\right\vert
^{p\left( x\right) }\vartheta \left( x\right) dx-\frac{1}{q^{-}}%
\dint\limits_{\Omega }\left\vert f\right\vert ^{q\left( x\right) }\vartheta
_{0}\left( x\right) dx \\
&\geq &\frac{1}{p^{+}}\dint\limits_{\Omega }\left\vert \nabla f\right\vert
^{p\left( x\right) }\vartheta \left( x\right) dx-\frac{1}{q^{-}}\max \left\{
\left\Vert f\right\Vert _{L^{q(.)}\left( \Omega ,\vartheta _{0}\right)
}^{q^{-}},\left\Vert f\right\Vert _{L^{q(.)}\left( \Omega ,\vartheta
_{0}\right) }^{q^{+}}\right\} \\
&\geq &\frac{1}{p^{+}}\left\Vert \nabla f\right\Vert _{L^{p\left( .\right)
}\left( \Omega ,\vartheta \right) }^{p^{-}}-\frac{1}{q^{-}}\max \left\{
\left\Vert f\right\Vert _{W_{0}^{1,q(.),p(.)}\left( \Omega ,\vartheta
_{0},\vartheta \right) }^{q^{-}},\left\Vert f\right\Vert
_{W_{0}^{1,q(.),p(.)}\left( \Omega ,\vartheta _{0},\vartheta \right)
}^{q^{+}}\right\} \\
&=&\frac{1}{p^{+}}\left\Vert f\right\Vert _{W_{0}^{1,q(.),p(.)}\left( \Omega
,\vartheta _{0},\vartheta \right) }^{p^{-}}-\frac{1}{q^{-}}\left\Vert
f\right\Vert _{W_{0}^{1,q(.),p(.)}\left( \Omega ,\vartheta _{0},\vartheta
\right) }^{q^{+}}.
\end{eqnarray*}%
Since $q^{+}<p^{-},$ we have $J\left( f\right) \longrightarrow \infty $ as $%
\left\Vert f\right\Vert _{W_{0}^{1,q(.),p(.)}\left( \Omega ,\vartheta
_{0},\vartheta \right) }\longrightarrow \infty .$ This completes the proof.
\end{proof}

\begin{theorem}
\label{l.s.c.}Let $\Omega \subset 
\mathbb{R}
^{d}$ be a bounded open set. Then the energy functional $J$ is weakly lower
semicontinuous.
\end{theorem}

\begin{proof}
Let $\left( f_{n}\right) _{n\in 
\mathbb{N}
}$ be a sequence of functions in $W_{0}^{1,q(.),p(.)}\left( \Omega
,\vartheta _{0},\vartheta \right) $ converging weakly to $f\in
W_{0}^{1,q(.),p(.)}\left( \Omega ,\vartheta _{0},\vartheta \right) .$ If we
consider the Theorem \ref{comson} and Theorem \ref{pro6}, then we get%
\begin{eqnarray*}
J\left( f\right) &=&\dint\limits_{\Omega }\frac{1}{p\left( x\right) }%
\left\vert \nabla f\right\vert ^{p\left( x\right) }\vartheta \left( x\right)
dx-\dint\limits_{\Omega }\frac{1}{q\left( x\right) }\left\vert f\right\vert
^{q\left( x\right) }\vartheta _{0}\left( x\right) dx \\
&\leq &\liminf_{n\longrightarrow \infty }J\left( f_{n}\right) .
\end{eqnarray*}%
That is the desired result.
\end{proof}

\begin{corollary}
If we consider the Theorem \ref{coercive} and Theorem \ref{l.s.c.}, then we
get that $J$ has a minimum point $f$ in $W_{0}^{1,q\left( .\right) ,p\left(
.\right) }\left( \Omega ,\vartheta _{0},\vartheta \right) $, i.e., $f$ is a
weak solution of (\ref{P1}), see \cite{RadRe}, \cite{Wil}.
\end{corollary}

\begin{theorem}
\label{PScon}The operator $J$ satisfies the (PS) condition.
\end{theorem}

\begin{proof}
Let $\left( f_{n}\right) _{n\in 
\mathbb{N}
}$ is a (PS) sequence, i.e.,%
\begin{equation*}
\left\vert J\left( f_{n}\right) \right\vert \leq M
\end{equation*}%
and%
\begin{equation*}
J^{\prime }\left( f_{n}\right) \longrightarrow 0\text{ in }%
W_{0}^{-1,q^{\prime }(.),p^{\prime }(.)}\left( \Omega ,\vartheta _{0}^{\ast
},\vartheta ^{\ast }\right)
\end{equation*}%
where $\frac{1}{p\left( .\right) }+\frac{1}{p^{\prime }\left( .\right) }=1,$ 
$\frac{1}{q\left( .\right) }+\frac{1}{q^{\prime }\left( .\right) }=1,$ $%
\vartheta _{0}^{\ast }=\vartheta _{0}^{1-q\left( .\right) }$ and $\vartheta
^{\ast }=\vartheta ^{1-p\left( .\right) }.$ Now, we want to prove that $%
\left( f_{n}\right) $ has a convergence subsequence. First, we will show
that $\left( f_{n}\right) $ is bounded in $W_{0}^{1,q(.),p(.)}\left( \Omega
,\vartheta _{0},\vartheta \right) .$ To see this, we assume that $\left(
f_{n}\right) $ is not bounded. Hence, we can suppose that $\left\Vert
f_{n}\right\Vert _{W_{0}^{1,q(.),p(.)}\left( \Omega ,\vartheta
_{0},\vartheta \right) }>1$ for all $n\in 
\mathbb{N}
.$ This follows that%
\begin{eqnarray*}
&&M+\left\Vert f_{n}\right\Vert _{W_{0}^{1,q(.),p(.)}\left( \Omega
,\vartheta _{0},\vartheta \right) } \\
&\geq &J\left( f_{n}\right) -\frac{1}{\lambda }\left\langle J^{\prime
}\left( f_{n}\right) ,f_{n}\right\rangle \\
&=&\dint\limits_{\Omega }\frac{1}{p\left( x\right) }\left\vert \nabla
f_{n}\right\vert ^{p\left( x\right) }\vartheta \left( x\right)
dx-\dint\limits_{\Omega }\frac{1}{q\left( x\right) }\left\vert
f_{n}\right\vert ^{q\left( x\right) }\vartheta _{0}\left( x\right) dx \\
&&-\frac{1}{\lambda }\dint\limits_{\Omega }\left\vert \nabla
f_{n}\right\vert ^{p\left( x\right) }\vartheta \left( x\right) dx+\frac{1}{%
\lambda }\dint\limits_{\Omega }\left\vert f_{n}\right\vert ^{q\left(
x\right) }\vartheta _{0}\left( x\right) dx \\
&\geq &\left( \frac{1}{p^{+}}-\frac{1}{\lambda }\right) \dint\limits_{\Omega
}\left\vert \nabla f_{n}\right\vert ^{p\left( x\right) }\vartheta \left(
x\right) dx+\left( \frac{1}{\lambda }-\frac{1}{q^{-}}\right)
\dint\limits_{\Omega }\left\vert f_{n}\right\vert ^{q\left( x\right)
}\vartheta _{0}\left( x\right) dx \\
&\geq &\left( \frac{1}{p^{+}}-\frac{1}{\lambda }\right) \left\Vert
f_{n}\right\Vert _{W_{0}^{1,q(.),p(.)}\left( \Omega ,\vartheta
_{0},\vartheta \right) }^{p^{-}}+\left( \frac{1}{\lambda }-\frac{1}{q^{-}}%
\right) \left\Vert f_{n}\right\Vert _{W_{0}^{1,q(.),p(.)}\left( \Omega
,\vartheta _{0},\vartheta \right) }^{q^{+}}
\end{eqnarray*}%
or equivalently%
\begin{eqnarray*}
&&M+\left\Vert f_{n}\right\Vert _{W_{0}^{1,q(.),p(.)}\left( \Omega
,\vartheta _{0},\vartheta \right) }+\left( \frac{1}{q^{-}}-\frac{1}{\lambda }%
\right) \left\Vert f_{n}\right\Vert _{W_{0}^{1,q(.),p(.)}\left( \Omega
,\vartheta _{0},\vartheta \right) }^{q^{+}} \\
&\geq &\left( \frac{1}{p^{+}}-\frac{1}{\lambda }\right) \left\Vert
f_{n}\right\Vert _{W_{0}^{1,q(.),p(.)}\left( \Omega ,\vartheta
_{0},\vartheta \right) }^{p^{-}}
\end{eqnarray*}%
Therefore, we have $\lambda \leq p^{+}$ which is a contradiction. This means
that $\left( f_{n}\right) $ is bounded in $W_{0}^{1,q(.),p(.)}\left( \Omega
,\vartheta _{0},\vartheta \right) .$ By this boundedness, there exists a
subsequence $f\in W_{0}^{1,q(.),p(.)}\left( \Omega ,\vartheta _{0},\vartheta
\right) $ such that $f_{n}\rightharpoonup f$ in $W_{0}^{1,q(.),p(.)}\left(
\Omega ,\vartheta _{0},\vartheta \right) .$ Now, we will show that there is
a subsequence $f\in W_{0}^{1,q(.),p(.)}\left( \Omega ,\vartheta
_{0},\vartheta \right) $ such that $f_{n}\longrightarrow f$ in $%
W_{0}^{1,q(.),p(.)}\left( \Omega ,\vartheta _{0},\vartheta \right) .$ Since $%
W_{0}^{1,q\left( .\right) ,p\left( .\right) }\left( \Omega ,\vartheta
_{0},\vartheta \right) \hookrightarrow \hookrightarrow L^{q\left( .\right)
}\left( \Omega ,\vartheta _{0}\right) $ holds, it is clear that $%
\left\langle J^{\prime }\left( f_{n}\right) ,f_{n}-f\right\rangle
\longrightarrow 0$ as $n\longrightarrow \infty .$ Moreover, we have%
\begin{eqnarray}
\left\langle J^{\prime }\left( f_{n}\right) ,f_{n}-f\right\rangle
&=&\dint\limits_{\Omega }\left\vert \nabla f_{n}\right\vert ^{p\left(
x\right) }\vartheta \left( x\right) dx-\dint\limits_{\Omega }\left\vert
\nabla f_{n}\right\vert ^{p\left( x\right) -2}\nabla f_{n}\nabla f\vartheta
\left( x\right) dx  \notag \\
&&-\dint\limits_{\Omega }\left\vert f_{n}\right\vert ^{q\left( x\right)
-2}f_{n}\left( f_{n}-f\right) \vartheta _{0}\left( x\right) dx.  \label{Jfn}
\end{eqnarray}%
Again, if we consider $W_{0}^{1,q\left( .\right) ,p\left( .\right) }\left(
\Omega ,\vartheta _{0},\vartheta \right) \hookrightarrow \hookrightarrow
L^{q\left( .\right) }\left( \Omega ,\vartheta _{0}\right) ,$ we get that $%
f_{n}\longrightarrow f$ in $L^{q\left( .\right) }\left( \Omega ,\vartheta
_{0}\right) .$ This follows by \cite{Sam1} that%
\begin{eqnarray*}
&&\left\vert \dint\limits_{\Omega }\left\vert f_{n}\right\vert ^{q\left(
.\right) -2}f_{n}\left( f_{n}-f\right) \vartheta _{0}\left( x\right)
dx\right\vert \\
&\leq &C_{1}\left\Vert \vartheta _{0}\right\Vert _{L^{\alpha \left( .\right)
}\left( \Omega \right) }\left\Vert \left\vert f_{n}\right\vert ^{q\left(
.\right) -1}\right\Vert _{L^{\beta \left( .\right) }\left( \Omega ,\vartheta
_{0}\right) }\left\Vert f_{n}-f\right\Vert _{L^{q\left( .\right) }\left(
\Omega ,\vartheta _{0}\right) } \\
&\leq &C_{2}\left\Vert f_{n}\right\Vert _{W_{0}^{1,q\left( .\right) ,p\left(
.\right) }\left( \Omega ,\vartheta _{0},\vartheta \right) }\left\Vert
f_{n}-f\right\Vert _{L^{q\left( .\right) }\left( \Omega ,\vartheta
_{0}\right) }\longrightarrow 0
\end{eqnarray*}%
where $\frac{1}{\alpha \left( .\right) }+\frac{1}{\beta \left( .\right) }+%
\frac{1}{q\left( .\right) }=1.$ By (\ref{Jfn}), we have%
\begin{equation*}
\lim_{n\longrightarrow \infty }\left( \dint\limits_{\Omega }\left\vert
\nabla f_{n}\right\vert ^{p\left( x\right) }\vartheta \left( x\right)
dx-\dint\limits_{\Omega }\left\vert \nabla f_{n}\right\vert ^{p\left(
x\right) -2}\nabla f_{n}\nabla f\vartheta \left( x\right) dx\right) =0.
\end{equation*}%
Now, we denote $\Lambda \left( f\right) =\dint\limits_{\Omega }\frac{1}{%
p\left( x\right) }\left\vert \nabla f\right\vert ^{p\left( x\right)
}\vartheta \left( x\right) dx$, for convenience. It is obvious that $\Lambda 
$ is a convex functional. Then there exist a $t\in \left[ 0,1\right] $ such
that%
\begin{eqnarray*}
\frac{\Lambda \left( g+t\left( f-g\right) \right) -\Lambda \left( g\right) }{%
t} &\leq &\frac{\left( 1-t\right) \Lambda \left( g\right) +t\Lambda \left(
f\right) -\Lambda \left( g\right) }{t} \\
&=&\Lambda \left( f\right) -\Lambda \left( g\right) .
\end{eqnarray*}%
This yields%
\begin{eqnarray*}
\left\langle \Lambda ^{\prime }\left( g\right) ,f-g\right\rangle
&=&\lim_{t\longrightarrow 0}\frac{\Lambda \left( g+t\left( f-g\right)
\right) -\Lambda \left( g\right) }{t} \\
&\leq &\Lambda \left( f\right) -\Lambda \left( g\right)
\end{eqnarray*}%
or equivalently%
\begin{equation*}
\left\langle \Lambda ^{\prime }\left( g\right) ,f-g\right\rangle \leq
\Lambda \left( f\right) -\Lambda \left( g\right)
\end{equation*}%
This follows that%
\begin{equation}
0=\lim_{n\longrightarrow \infty }\left\langle \Lambda ^{\prime }\left(
f_{n}\right) ,f-f_{n}\right\rangle \leq \Lambda \left( f\right)
-\lim_{n\longrightarrow \infty }\Lambda \left( f_{n}\right) .
\label{altyari}
\end{equation}%
Moreover, it is easy to see that $\Lambda $ is weakly lower semicontinuous.
This follows by (\ref{altyari}) that%
\begin{equation*}
\lim_{n\longrightarrow \infty }\Lambda \left( f_{n}\right) =\Lambda \left(
f\right) .
\end{equation*}%
Now, we are ready to prove that $f_{n}\longrightarrow f$ in $%
W_{0}^{1,q(.),p(.)}\left( \Omega ,\vartheta _{0},\vartheta \right) .$ Assume
that the sequence $\left( f_{n}\right) $ is not convergent to $\left(
f\right) $ in $W_{0}^{1,q(.),p(.)}\left( \Omega ,\vartheta _{0},\vartheta
\right) .$ Thus, for $\varepsilon _{1}>0,$ there exists a subsequence $%
\left( f_{n_{k}}\right) $ of $\left( f_{n}\right) $ such that $\left\Vert
f_{n_{k}}-f\right\Vert _{W_{0}^{1,q\left( .\right) ,p\left( .\right) }\left(
\Omega ,\vartheta _{0},\vartheta \right) }\geq \varepsilon _{1}.$ Since $%
\Lambda $ is convex functional, we have%
\begin{equation}
\limsup_{n\longrightarrow \infty }\Lambda \left( \frac{f_{n_{k}}+f}{2}%
\right) \leq \Lambda \left( f\right) .  \label{limsup}
\end{equation}%
Moreover, it is clear that $\left( \frac{f_{n_{k}}+f}{2}\right)
\rightharpoonup f$ in $W_{0}^{1,q(.),p(.)}\left( \Omega ,\vartheta
_{0},\vartheta \right) .$ Since $\Lambda $ is weakly lower semicontinuous,
we have%
\begin{equation*}
\Lambda \left( f\right) \leq \liminf_{n\longrightarrow \infty }\Lambda
\left( \frac{f_{n_{k}}+f}{2}\right)
\end{equation*}%
which is a contradiction in sense to (\ref{limsup}). This follows that $%
f_{n}\longrightarrow f$ in $W_{0}^{1,q(.),p(.)}\left( \Omega ,\vartheta
_{0},\vartheta \right) .$ This completes the proof.
\end{proof}

\begin{theorem}
Let $p^{+}<q^{-}<\lambda .$ Then, the Problem (P1) has a nontrivial weak
solution.
\end{theorem}

\begin{proof}
For this theorem, our motivation is based on Mountain Pass Theorem (see \cite%
{Wil}). By Theorem \ref{PScon}, the energy functional $J$ satisfies (PS)
condition on $W_{0}^{1,q(.),p(.)}\left( \Omega ,\vartheta _{0},\vartheta
\right) .$ If we consider the \cite[Proposition 2.4]{Liu} (or \cite{A2}),
then we have%
\begin{eqnarray*}
J\left( f\right) &=&\dint\limits_{\Omega }\frac{1}{p\left( x\right) }%
\left\vert \nabla f\right\vert ^{p\left( x\right) }\vartheta \left( x\right)
dx-\dint\limits_{\Omega }\frac{1}{q\left( x\right) }\left\vert f\right\vert
^{q\left( x\right) }\vartheta _{0}\left( x\right) dx \\
&\geq &\frac{1}{p^{+}}\dint\limits_{\Omega }\left\vert \nabla f\right\vert
^{p\left( x\right) }\vartheta \left( x\right) dx-\frac{1}{q^{-}}%
\dint\limits_{\Omega }\left\vert f\right\vert ^{q\left( x\right) }\vartheta
_{0}\left( x\right) dx \\
&\geq &\frac{1}{p^{+}}\left\Vert f\right\Vert _{W_{0}^{1,q(.),p(.)}\left(
\Omega ,\vartheta _{0},\vartheta \right) }^{p^{+}}-\frac{1}{q^{-}}\left\Vert
f\right\Vert _{W_{0}^{1,q(.),p(.)}\left( \Omega ,\vartheta _{0},\vartheta
\right) }^{q^{-}} \\
&\geq &\left( \frac{1}{p^{+}}-\frac{1}{q^{-}}\right) \left\Vert f\right\Vert
_{W_{0}^{1,q(.),p(.)}\left( \Omega ,\vartheta _{0},\vartheta \right)
}^{q^{-}}>0
\end{eqnarray*}%
for $\left\Vert f\right\Vert _{W_{0}^{1,q(.),p(.)}\left( \Omega ,\vartheta
_{0},\vartheta \right) }\leq 1.$ Thus, when $\left\Vert f\right\Vert
_{W_{0}^{1,q(.),p(.)}\left( \Omega ,\vartheta _{0},\vartheta \right) }=\rho $
sufficiently small, we have $J\left( f\right) >0.$ Moreover, since $%
p^{+}<q^{-}$, we get%
\begin{eqnarray*}
J\left( tg\right) &=&\dint\limits_{\Omega }\frac{1}{p\left( x\right) }%
\left\vert t\nabla g\right\vert ^{p\left( x\right) }\vartheta \left(
x\right) dx-\dint\limits_{\Omega }\frac{1}{q\left( x\right) }\left\vert
tg\right\vert ^{q\left( x\right) }\vartheta _{0}\left( x\right) dx \\
&\leq &t^{p^{+}}\dint\limits_{\Omega }\left\vert \nabla g\right\vert
^{p\left( x\right) }\vartheta \left( x\right)
dx-t^{q^{-}}\dint\limits_{\Omega }\left\vert g\right\vert ^{q\left( x\right)
}\vartheta _{0}\left( x\right) dx\longrightarrow -\infty
\end{eqnarray*}%
as $t\longrightarrow \infty $ for $g\in W_{0}^{1,q(.),p(.)}\left( \Omega
,\vartheta _{0},\vartheta \right) -\left\{ 0\right\} .$ It is note that $%
J\left( 0\right) =0.$ This follows that $J$ satisfies the geometric
conditions of the Mountain Pass Theorem (see \cite{Calo}, \cite{Napo}, \cite%
{Wil}), and the operator $J$ admits at least one nontrivial critical point.
\end{proof}

\end{document}